\documentclass[reqno,12pt]{amsart}

\usepackage[utf8]{inputenc}

\usepackage{enumerate}
\usepackage[margin=1in,marginparwidth=0.7in]{geometry}

\usepackage{ifpdf}
\usepackage{amsmath}
\usepackage{amsfonts}
\usepackage{amssymb}
\usepackage{amsthm}
\usepackage[ocgcolorlinks,hyperfootnotes=false,colorlinks=true,citecolor=blue,linkcolor=blue,urlcolor=blue]{hyperref}
\usepackage{setspace}
\usepackage{amsrefs}
\usepackage{nicefrac}
\usepackage{graphicx}
\usepackage{color}
\usepackage{mathtools}

\newcommand{\ignore}[1]{}

\renewcommand{\Re}{\operatorname{Re}}
\renewcommand{\Im}{\operatorname{Im}}

\newcommand{\abs}[1]{\left\lvert {#1} \right\rvert}
\newcommand{\sabs}[1]{\lvert {#1} \rvert}

\newcommand{\norm}[1]{\left\lVert {#1} \right\rVert}
\newcommand{\snorm}[1]{\lVert {#1} \rVert}
\newcommand{\bnorm}[1]{\bigl\lVert {#1} \bigr\rVert}

\newcommand{\C}{{\mathbb{C}}}
\newcommand{\R}{{\mathbb{R}}}

\newcommand{\D}{{\mathbb{D}}}

\newcommand{\sV}{{\mathcal{V}}}

\newcommand{\rank}{\operatorname{rank}}

\newtheorem{thm}{Theorem}[section]

\newtheorem{prop}[thm]{Proposition}

\newtheorem{lemma}[thm]{Lemma}

\newtheorem{question}[thm]{Question}

\theoremstyle{definition}
\newtheorem{defn}[thm]{Definition}
\newtheorem{example}[thm]{Example}

\theoremstyle{remark}

\author{Ji\v{r}\'{\i} Lebl}
\thanks{The first author was in part supported by Simons Foundation collaboration grant 710294.}
\address{Department of Mathematics, Oklahoma State University,
Stillwater, OK 74078, USA}
\email{lebl@okstate.edu}

\author{Alan Noell}
\address{Department of Mathematics, Oklahoma State University,
Stillwater, OK 74078, USA}
\email{alan.noell@okstate.edu}

\author{Sivaguru Ravisankar}
\address{Tata Institute of Fundamental Research, Centre for Applicable Mathematics, Bengaluru 560065, India}
\email{sivaguru@tifrbng.res.in}

\date{August 8, 2021}

\ifpdf
\hypersetup{
  pdftitle={On CR singular CR images},
  pdfauthor={Jiri Lebl, Alan Noell, Sivaguru Ravisankar},
  pdfsubject={Several Complex Variables},
  pdfkeywords={CR singular, removable CR singularity, CR image},
}
\fi

\title{On CR singular CR images}


\begin{document}

\begin{abstract}
We say that a CR singular submanifold $M$ has a removable
CR singularity if the CR structure at the CR points of $M$
extends through the
singularity as an abstract CR structure on $M$.  We study such real-analytic
submanifolds, in which case removability is equivalent to $M$ being the
image of a generic real-analytic submanifold $N$ under a holomorphic map that is
a diffeomorphism of $N$ onto $M$, what we call a CR image.
We study the stability of the CR singularity under perturbation,
the associated quadratic invariants,
and
conditions for removability of a CR singularity.
A lemma is also proved about perturbing away the zeros of holomorphic functions on CR submanifolds, which could be of independent interest.
\end{abstract}

\maketitle


\section{Introduction} \label{section:intro}

Let $M \subset \C^m$ be a real-analytic submanifold, and let $T^{0,1}_p M$
be the set of antiholomorphic tangent vectors at $p \in M$,
that is, vectors in $\C \otimes T_p M$ that are in the span of
$\frac{\partial}{\partial \bar{z}_1},\ldots,\frac{\partial}{\partial \bar{z}_m}$.
If the dimension of $T^{0,1}_p M$ is constant as $p$ varies over $M$
(we assume $M$ is connected),
then $M$ is said to be a \emph{CR submanifold}, and
the dimension of $T^{0,1}_p M$ is the \emph{CR dimension} of $M$.
The disjoint union $T^{0,1} M$ then becomes a real-analytic vector bundle.
Any real hypersurface is a CR submanifold,
but submanifolds of higher codimension need not be so.
If $T^{0,1}_p M$ is not constant in any neighborhood of $q \in M$,
we say $q$ is a \emph{CR singularity}.
If $T^{0,1}_p M$ is constant near $q \in M$, then $q$ is a \emph{CR point}.
We denote the set of CR points of $M$ as $M_{CR}$.
The purpose of this paper is to study certain CR singularities.  Namely:

\begin{defn}
A point $q$ of a real-analytic submanifold $M \subset \C^m$ is a
\emph{removable CR singularity} if there is,
locally near $q$, a real-analytic vector bundle $\sV$ on $M$ such that, for every
CR point $p$, the fiber $\sV_p$ equals $T^{0,1}_p M$.

If $M$ has only removable CR singularities, we say it is a \emph{CR image},
and we call $(M,\sV)$ the
\emph{extended abstract CR structure}.
\end{defn}

As $p \mapsto \dim T_p^{0,1} M$ is given by the rank of a certain matrix with
real-analytic entries, the set of CR singularities is a
real-analytic subvariety of $M$. Furthermore, if $M$ is connected,
then the CR dimension of $M_{CR}$ is constant.
The vector bundle $\sV$ from above
defines an abstract CR structure as it satisfies the
conditions $\sV \cap \bar{\sV} = \{ 0 \}$ and 
$[\sV,\sV] \subset \sV$ on a dense open set, and therefore everywhere.
In other words, a CR singularity $q$ is removable if the manifold $M$ is an
abstract CR manifold near $q$ whose CR structure equals the CR structure
induced by the embedding at CR points.  Note that our definition includes CR points as removable CR singularities.

The smallest complex subvariety containing a submanifold $N \subset \C^n$
is called the \emph{intrinsic complexification} of $N$. For CR submanifolds, the
intrinsic complexification is a complex submanifold.
We say a real-analytic submanifold $N$ of $\C^n$ is \emph{generic}
if $N$ is CR and the intrinsic
complexification of $N$ is (locally) equal to $\C^n$.  In our context,
this is equivalent to the standard definition of generic.
If the CR dimension
of $N$ is $d$ and the real codimension of $N$ in $\C^n$ is $k$, then
$N$ is generic if and only if $k=n-d$.
More specifically, the intrinsic complexification is always of dimension
$d+k$.

Let $M$ be a CR singular submanifold with a removable CR singularity.
The abstract CR manifold $(M,\sV)$ is integrable, so it 
can be (locally) realized as a generic submanifold $N$ in $\C^n$ (where 
$n=\dim M - \operatorname{CRdim} M$)
(see \cite{BER:book}*{Theorem 2.1.11}).  This leads to the following proposition.

\begin{prop}
A point $q$ of a real-analytic submanifold $M \subset \C^m$ is a
removable CR singularity if and only if
there exist a real-analytic generic submanifold
$N \subset \C^n$, $n \leq m$, a neighborhood $U$ of $N$ in $\C^n$,
a neighborhood $W$ of $q$ in $\C^m$,
and a holomorphic mapping $F \colon U \to W$ such that
\begin{enumerate}[(i)]
\item The (holomorphic) Jacobian $DF$ is of rank $n$ on a relatively
open dense subset of $N$,
\item $M \cap W = F(N)$,
\item $F|_N$ is a diffeomorphism onto its image.
\end{enumerate}
\end{prop}

The forward implication is proved by constructing the generic $N$ from $(M,\sV)$, and then extending the
natural map $f \colon N \to M$ into the intrinsic complexification of $N$ to get a holomorphic map $F$.
On $f^{-1}(M_{CR})$, $F$ will locally be a biholomorphism
onto the intrinsic
complexification of $M$ (it is a real-analytic CR diffeomorphism of CR
submanifolds there).  This means that $F$ is of rank $n$ on
$f^{-1}(M_{CR})$, which is a relatively open dense subset of $N$.

Due to the proposition, we may say that $M$ is a CR image of $N$.
As our manifolds are real analytic, we can just think of the embedded $N$ from
the proposition as the extended abstract CR structure.
CR images were first noticed in Ebenfelt--Rothschild~\cite{ER} and were
studied more generally in Lebl--Minor--Shroff--Son--Zhang~\cite{LMSSZ}.
The extended structure is the CR Nash blowup; see Garrity~\cite{Garrity}.

We consider three main questions regarding CR images.
First we study whether, given a CR image, a CR singularity can be perturbed away
by keeping the CR structure $N$ fixed and perturbing the embedding $F$.
In Section~\ref{sec:equidim},
we find a geometric condition for stability of the CR singularity
when $n = m$.
In the equidimensional setting,
every CR singularity of $M$ is the image under $F$ of a point $p\in N$ where the Jacobian determinant 
of $F$ vanishes (see Proposition \ref{prop:locCRsing}).
We require that the variety where the Jacobian determinant
of $F$ vanishes, $Z_{\det DF}$, 
be ``transverse'' to the tangent space of $N$
in the following way:
\begin{equation}
H_p N \not\subset
 C_p Z_{\det DF} .
\end{equation}
Here, $H_p$ is the set of tangent vectors invariant under
multiplication by $i$, which is a complex vector space,
and $C_p$ is the tangent cone.
Under these conditions,  a CR singularity at $F(p)$ cannot be eliminated by perturbing $F$,
that is, for all $G$ close enough to $F$, $G(N)$ is CR singular.
See Theorem \ref{thm:noperturbmapequi}.
So in general a CR singularity cannot be eliminated
in this way.

This theorem follows from an interesting result (Lemma~\ref{lemma:main}) about perturbing away
zeros of holomorphic functions on CR submanifolds.  In particular, the
transversality condition above guarantees that the zero set cannot be
perturbed away.  This problem is surprisingly subtle, and
a more general version of it
leads to difficult questions involving homotopy groups of
spheres (see \cite{CoffmanLebl}).
What we prove is for $N$ as above:  We cannot perturb away the zero of
a holomorphic function $\varphi$
as long as
\begin{equation}
H_p N \not\subset
 C_p Z_{\varphi} .
\end{equation}
Not only can $Z_{\varphi}$ be singular, but even if it is a manifold
the derivatives of $\varphi$ could vanish at $p$.  Moreover, this result is
not true if $\varphi$ is only be assumed to be real analytic, even if the
zero set of $\varphi$ is a submanifold and transverse to $N$ in the standard
way.  For example, if $\varphi$ is a real square, then adding
$\epsilon > 0$ eliminates the zero completely.
We remark that 
$H_p N \not\subset C_p Z_{\varphi}$ in particular implies that
$\varphi$ does not vanish identically on $N$.

Returning to CR images, in the non-equidimensional setting,
we find an inequality in terms of $n$, $m$,
and the codimension $k$ of $N$ that guarantees that a CR singularity can be perturbed away.
In particular, if
\begin{equation} \label{eq:dimineq}
4n-k < 2(m+1) ,
\end{equation}
then near every point of $N$ there exists $G$ arbitrarily close to $F$
such that $G(N)$ is CR.
See Theorem \ref{thm:perturbmap}.

The inequality \eqref{eq:dimineq} is sharp:
For any $2 \leq k \leq n \leq m$ such that
$4n-k \geq 2(m+1)$, there exist $N$ and $F$ such that $F(N)$ is CR singular and
$G(N)$ is CR singular
for all small enough perturbations $G$ of $F$.  See
Proposition~\ref{prop:sharpnessconstruction}.  In fact, for
these dimensions, such an example exists near any CR image.

The second question we consider, in Section \ref{section:quad},
concerns
the biholomorphic invariants found
in \cite{crext5} for
codimension-2 CR images. We show that they are independent of the
extended abstract CR structure, that is, we can pair any quadratic invariant
with any abstract CR structure.  We also show by example that there  exist
higher-order invariants of the embedding that are independent of the
extended abstract CR structure.

Our third question concerns codimension-2 CR singular submanifolds $M$ of $\C^3$ for which the CR dimension of $M_{CR}$ is 1.
See Section~\ref{sec:dim3codim1}.
The CR structure is given by a single vector field outside of the CR singularities, and we consider the quotients of the two coefficients of this vector field. 
A necessary and sufficient condition for a CR singularity to be
removable is that one of the quotients extend to be real analytic at the singularity.  
There exist examples of submanifolds with a 
CR vector field, nonvanishing at the CR singular points,
that is $C^k$, but no better. 
If such a smooth vector field exists,
then there exists a real-analytic one,
and the CR singularity is removable.  Furthermore, we prove that
submanifolds with nonvanishing quadratic invariants (as in the preceding section)
that do not correspond to a parabolic Bishop surface and have a small enough
CR singular set are not CR images.

CR singular submanifolds were first studied by Bishop~\cite{Bishop65} in
$\C^2$.  His work was extended in $\C^2$ by 
Moser--Webster~\cite{MoserWebster83},
Moser~\cite{Moser85},
Kenig--Webster
\cites{KenigWebster:82}, Gong~\cite{Gong94:duke},
Huang--Krantz~\cite{HuangKrantz95}, 
Huang--Yin~\cite{HuangYin09}, 
and many others.
In dimension higher than 2, the work is more recent.
See Lai~\cite{Lai}, Dolbeault--Tomassini--Zaitsev~\cites{DTZ,DTZ2},
Coffman~\cite{Coffman},
Huang~\cite{Huang98},
Huang--Yin~\cites{HuangYin09:codim2,HuangYin:flattening1,HuangYin:flattening2},
Fang--Huang~\cite{FangHuang},
Burcea~\cite{Burcea},
Gong--Lebl~\cite{GongLebl},
Gong--Stolovich~\cites{GongStolovitchI,GongStolovitchII},
Slapar~\cite{Slapar},
Slapar--Star\v{c}i\v{c}~\cite{SlaparStarcic},
and the references within.


\section{Perturbing away zeros of holomorphic functions on CR submanifolds}

We first take the promised detour on perturbing away zeros of holomorphic
functions on CR submanifolds.  More precisely, we give conditions when this perturbation is
not possible.

As we mentioned before, any 
abstract real-analytic CR structure on an abstract real-analytic manifold
is (locally) integrable (see \cite{BER:book}*{Theorem 2.1.11}).  Thus, as far as local results are concerned, we may always just study
generic submanifolds in $\C^n$ if we are studying real-analytic CR
manifolds.

Suppose that 
$N \subset \C^n$ is a  real-analytic generic submanifold and $\varphi$ is a holomorphic
function that is zero somewhere on $N$.
The question we wish to study in this section is:
Can $\varphi$ be perturbed to move the zero off of $N$?
If we think of the zero set of $\varphi$,
the set $Z_\varphi$, then the idea that comes to mind is transversality.
If $N$ and $Z_\varphi$ meet transversally, that is, if their tangent spaces
at the intersection span the tangent space of $\C^n$, then
perturbing $Z_\varphi$ as a complex submanifold cannot eliminate the
intersection.
However, we are interested in perturbing not $Z_\varphi$
but the function $\varphi$.

If $N$ and $Z_\varphi$ were
both complex analytic, then a result like this would simply boil down to 
Rouch\'e's theorem, where the ``transversality'' condition would simply be
that $Z_\varphi$ does not contain $N$.

Interestingly, 
as long as the tangent cone of
$Z_\varphi$ and the complex tangent space of $N$ are ``transverse'' in the right
sense, then again zeros cannot be perturbed away.  The examples after the
proof of Lemma \ref{lemma:main} show why some transversality condition is necessary.
Recall that the \emph{tangent cone} $C_p X$ of a complex subvariety $X\subset \C^n$ at $p \in X$
is defined as the set of vectors $v$ such that there are a sequence $\{ p_j
\}$ in $X$ converging to $p$ and a sequence of complex numbers $\{ a_j \}$ such that
\begin{equation}
v = \lim_{j\to \infty} a_j (p_j - p) .
\end{equation}
We regard $C_p X$ as a subspace of the tangent space of $\C^n$ at
$p$.
If $\varphi$ is holomorphic near $p$, it is not hard to see (e.g.,  \cite{Whitney:book}*{Chapter 7, Theorem 4A}) that the tangent cone of $Z_\varphi$ at $p$ is given by the vanishing of the
lowest-order terms of the expansion of $\varphi$ at $p$.
If $X$ is a complex submanifold near $p$, the tangent cone is simply
the tangent space at $p$.
The \emph{complex tangent space} $H_p N$ of a CR submanifold is precisely
the set of tangent vectors of $N$ invariant under multiplication by $i$.
As $H_p N$ has a natural complex structure, we regard $H_p N$ as a
complex vector space.  If $N$ is a complex submanifold, then
$T_p N = H_p N$, and we regard $T_p N$ as a complex vector space.
Transversality using the complex tangent space $H_p N$ is sometimes called
CR transversality, although we require a somewhat different notion.

The transversality condition we will consider is the following:
\begin{equation}
H_p N \not\subset C_p Z_\varphi .
\end{equation}
To justify this condition note that if $d\varphi \not= 0$ and $Z_\varphi$
is a submanifold, then $C_p Z_\varphi$ is the tangent space and is
$(n-1)$-dimensional.  Suppose for a moment that $N$ is a complex
submanifold. Then $H_p N \not\subset C_p Z_\varphi$ simply means 
standard transversality: $\operatorname{span} \{  T_p N , T_p Z_\varphi \} =
T_p \C^n$,
as at least one (complex) direction of $T_p N$ is
not in $T_p Z_\varphi$, and we need only one more direction.

Note that if $H_p N \not\subset C_p Z_\varphi$, then necessarily
$H_p N$ must be of positive dimension, that is, $N$ must be of CR dimension
at least 1.

\begin{lemma} \label{lemma:main}
Suppose $N \subset \C^n$ is a connected real-analytic generic submanifold, $p\in N$,
 $\varphi$ is a holomorphic function defined in a neighborhood of $p$, and
$\varphi(p)=0$.
If
\begin{equation}
H_p N \not\subset C_p Z_\varphi ,
\end{equation}
then
for any neighborhood $W$ of $p$ there exists $\epsilon > 0$ such that
whenever $\psi \colon W \to \C$ is holomorphic
and $\norm{\varphi-\psi}_{W} < \epsilon$, then $\psi$ vanishes somewhere on $N \cap W$.
\end{lemma}

To prove this lemma, we find a continuous family of analytic discs attached to $N$
shrinking to $p$ such that one of them intersects $Z_{\varphi}$. Let $\alpha \in (0,1)$. By a
$C^{1,\alpha}$ \emph{analytic disc} $\Delta \colon \overline{\D} \to \C^n$ \emph{attached to}
$N$ we mean that $\Delta$ is $C^{1,\alpha}$, is holomorphic in $\D$, and satisfies
$\Delta(\partial \D) \subset N$.  By a \emph{continuous family} $\Delta_t$ for $t
\in [0,1]$ we mean that $\Delta_t(\xi)$ is continuous as a function of
$(\xi,t) \in \overline{\D} \times [0,1]$.
Below $\alpha \in (0,1)$ is arbitrary but fixed.

\begin{lemma} \label{lemma:discs}
Suppose $N \subset \C^n$ is a connected real-analytic generic submanifold, $p\in N$,
 $\varphi$ is a holomorphic function defined in a neighborhood of $p$, and
$\varphi(p)=0$.
Suppose
\begin{equation}
H_p N \not\subset C_p Z_\varphi .
\end{equation}
Then there exists a continuous family of $C^{1,\alpha}$  analytic discs $\Delta_t
\colon \overline\D \to \C^n$ attached to $N$, for $t \in [0,1]$, such that $\Delta_0 \equiv
p$, $\Delta_t$ is nonconstant for $t > 0$,
and $\varphi \circ \Delta_t$ has an isolated zero in $\D$ for all sufficiently small positive $t$. Thus, $\Delta_t(\D)$ intersects
$Z_\varphi$  for such $t$.
\end{lemma}

\begin{proof}
We will use the fact that the set of $C^{1,\alpha}$ analytic discs attached to $N$
is a smooth Banach submanifold; see the work of
Tumanov~\cite{Tumanov:ext} and Baouendi-Rothschild-Tr\'epreau~\cite{BRT}.  For reference we use the results as given in
Baouendi-Ebenfelt-Rothschild \cite{BER:book}.

Assume $p=0$, and suppose $N$ is given in coordinates $(z,w)
\in \C^d \times \C^k$ by $\Im w = r(z,\bar{z},\Re w)$ where $dr(0)=0$, so $H_0 N$
can be identified with $\{ w = 0 \}$.
After a linear change of coordinates that preserves $\{ w = 0 \}$,
we can assume that the $z_1$-axis, that is, the line given
by $w=0$ and $z_2=\cdots=z_d = 0$, does not lie in $C_0 Z_\varphi$.
By restricting to the $(z_1,w_1,\ldots,w_k)$-space, we can just assume that
the CR dimension of $N$ is exactly 1, that is, $d=1$ above.

By Theorem 6.5.4 in \cite{BER:book} (see also Theorem 6.2.12),
we can find a continuous
family $\Delta_t$ of $C^{1,\alpha}$ analytic discs attached to $N$
of the form $\Delta_t(\xi) = \bigl(t \xi, g_t(\xi)\bigr) \in \C \times \C^k$ such that
$\Re g_t(1)=0$ for all $t$ and $g_0 \equiv 0$ (the limiting disc is constant).
In fact, $\Delta_t$ and hence $g_t(\xi)$ can be chosen to be
smooth
as a function of $(t,\xi)$.
As $g_0 \equiv 0$,
$g_t(\xi) = t h_t(\xi)$ for some smooth $h$.
As $\Delta_t$
is attached to $N$, for $u \in \partial \D$
\begin{equation}
\Im g_t(u) = r\bigl(tu,t\bar{u},\Re g_t(u)\bigr)
 = r\bigl(tu,t\bar{u},t \Re th_t(u)\bigr)
= t^2\gamma_t(\xi)
\end{equation}
for some smooth $\gamma$.
So $\Im h_0 \equiv 0$ on $\partial \D$,
and consequently by the maximum principle
$\Im h_0 \equiv 0$ on $\overline{\D}$.  Hence, $\Re h_0$ must be
constant.  Given that $\Re g_t(1) = 0$ for all $t$, we have that
$\Re h_0(1) = 0$, and so $\Re h_0 \equiv 0$.  In other words,
$g_t(\xi) = t^2 q_t(\xi)$ for some smooth $q$ holomorphic in $\xi$.
That is, the $w$ in the analytic disc vanishes to second
order.
(Another, perhaps more straightforward but less elementary, way to
see this fact is to consider the tangent space at the constant disc of the
Banach submanifold of analytic discs attached to $N$.)

Write $\varphi$ as
\begin{equation}
\varphi(z,w) = P(z) + R(z,w) + E(z,w),
\end{equation}
where $P+R$ is the initial homogeneous polynomial
(consisting of the lowest-order terms of $\varphi$ at $0$)
and $R(z,0) \equiv 0$ (every monomial in $R$ depends on $w$).
The zero set of $P+R$ is the tangent cone of $Z_\varphi$ at $0$. 
As $C_0
Z_\varphi$ does not include $\{ w = 0 \}$, we must have that $P$ is not
identically zero; without loss of generality, $P(z) = z^\ell$.
Plug $\Delta_t$ into $\varphi$:
\begin{equation}
\varphi \circ \Delta_t (\xi)
=
t^\ell \xi^\ell + R\bigl(t\xi,t^2q_t(\xi) \bigr) + E\bigl(t\xi,t^2q_t(\xi) \bigr)
=
t^\ell \xi^\ell + t^{\ell+1} H(t,\xi) ,
\end{equation}
for some smooth $H$ holomorphic in $\xi$.
We divide by $t^{\ell}$ and apply Rouch\'e's theorem to find that
$\varphi \circ \Delta_t (\xi)$ must have a zero in $\D$ for all
small enough positive $t$.  In other words, the analytic disc $\Delta_t$ intersects
$Z_\varphi$ for all small enough positive $t$.
\end{proof}

\begin{proof}[Proof of Lemma~\ref{lemma:main}]
Apply Lemma~\ref{lemma:discs} to find the family of discs. Then for some $t_0>0$ the function $\varphi \circ \Delta_{t}$ has an
isolated zero in $\D$ when $0<t\leq t_0$.  Because the discs
shrink down to $p$, by making $t_0$ smaller we may assume in addition that $\Delta_t(\overline{\D}) \subset W$
for all $t \leq t_0$.

As $\Delta_{t_0}$ is nonconstant, by Rouch\'e's theorem 
there exists $\epsilon > 0$ such that, for every
$\psi$ holomorphic in $W$ with $\snorm{\varphi-\psi}_W < \epsilon$,
$\psi \circ \Delta_{t_0}$ has an isolated zero in $\D$.

Then either $\psi(p) = 0$, in which case we are done, or for some
$0 < t < t_0$ the function $\psi\circ\Delta_{t}$ has a zero on $\partial \D$ (again by Rouch\'e).  As $\Delta_{t}(\partial \D) \subset N$,
it follows that $\psi$ vanishes on $N \cap W$.
\end{proof}

\begin{example}
First, let us show that $\varphi$ being holomorphic is necessary.
Let $(z,w)$ be the coordinates on $\C^2$, and
consider $N \subset \C^2$  given by $\Im w = 0$.  The CR dimension of $N$ is 1,
and $H_0N$ can be identified with the complex line $w=0$.
Take $\varphi$ to be the (non-holomorphic) function defined by 
$\varphi(z,w) = \abs{z}^2$.  Let $X=Z_\varphi$ be the zero set of $\varphi$.
The manifold $X$ meets $N$ transversally, so perturbing $X$ as a submanifold
would not get rid of the intersection with $N$.  However, perturbing
$\varphi$ using $\psi(z,w) = \abs{z}^2+\epsilon$ for $\epsilon > 0$
gets rid of all of the zeros of $\varphi$ (whether they are on $N$ or not).
That is, $Z_\psi = \emptyset$.
\end{example}

\begin{example}
Another instructive example is to consider the case when the CR dimension of $N$ is
0.  Then $Z_\varphi$ clearly cannot be ``transverse'' in our sense as
$H_0N = \{ 0 \}$.  It may seem possible to replace the transversality
requirement
by having $T_0N$ and $C_0 Z_\varphi$ span $T_0\C^n$ (taking either the real
span or the complex span). But consider $N = \R^2 \subset \C^2$ and the holomorphic function
$\varphi(z,w) = z^2+w^2$.  The vectors in
$T_0N$ and $C_0 Z_\varphi$ span $T_0\C^n$ as a real vector space.
However, $\psi = \varphi+\epsilon$ has no zeros on $N$ for $\epsilon > 0$.
\end{example}

More generally, we can consider the zero set of a $\C^{\nu}$-valued mapping
$\varphi$.  First, if the derivative is of full rank, then
the desired result is simply the standard result on transversality.
The simple version of Lemma~\ref{lemma:main} is the obvious one.

\begin{prop} \label{prop:normaltransversality}
Suppose $U \subset \C^n$ is open,
$N \subset U$ is a connected real-analytic submanifold, $p\in N$,
$\varphi \colon U \to \C^\nu$ is a holomorphic mapping
such that $D\varphi$ is of rank $\nu$ in $U$, $\varphi(p) = 0$,
and as real submanifolds $Z_\varphi$ and $N$ are transverse at $p$, that is,
\begin{equation}
\dim_{\R} \operatorname{span}_{\R}
\bigl\{ T_p Z_{\varphi} , T_p N \bigr\}
= 2n ,
\qquad \text{or equivalently} , \qquad
\operatorname{span}_{\R}
\bigl\{ T_p Z_{\varphi} , T_p N \bigr\}
= T_p \C^n .
\end{equation}
Then
for any neighborhood $W$ of $p$ there exists  $\epsilon > 0$ such that,
whenever $\psi \colon W \to \C^{\nu}$ is holomorphic
and $\norm{\varphi-\psi}_{W} < \epsilon$, then $\psi$ vanishes somewhere on $N \cap W$.
\end{prop}

\begin{proof}
We can apply the standard result on real transversality: Note that, as $D\varphi$ is
of rank $\nu$, for a small enough perturbation $\psi$ of $\varphi$,
by Cauchy's estimates the derivative $D\psi$
is still of rank $\nu$ on some slightly smaller neighborhood,
and $Z_\psi$ is a submanifold.
\end{proof}

We can also ask when we can generically perturb away the zero.  We will not
use the following proposition, although in Section \ref{section:nonequi} we will prove an analogous result
for removable CR singularities.

\begin{prop} \label{prop:genericallyperturbphi}
Suppose $U \subset \C^n$ is open,
$N \subset U$ is a connected real-analytic submanifold
of positive codimension, $p\in N$,
 $\varphi \colon U \to \C^\nu$ is a holomorphic mapping,
$\varphi(p) = 0$,  the
rank of $D\varphi$ is $\nu$ on a dense open subset of $U$, and
\begin{equation}
\nu\leq\frac{1}{2}\dim_{\R} N .
\end{equation}
Let $W \subset \subset U$ be a relatively compact neighborhood of $p$ and
$\delta > 0$.  Then there exists a holomorphic map $\psi \colon W \to \C^{\nu}$ such that
$\snorm{\varphi-\psi}_W < \delta$, $\psi(p) = 0$, and the zero set
of $\psi$ cannot be perturbed away near $p$.  That is,
for every $\epsilon > 0$ and every neighborhood $W' \subset W$ of $p$,
if $h \colon W' \to \C^\nu$ is holomorphic and $\snorm{h -\psi}_{W'} <
\epsilon$, then $h$ vanishes somewhere on $N \cap W'$.
\end{prop}

\begin{proof}
At points where the rank of $D\varphi$ is $\nu$, the level sets of $\varphi$
are complex submanifolds of complex dimension $n-\nu$.
As $\nu$ is at most
$\frac{1}{2}\dim_{\R} N$ by hypothesis, at such points
the real codimension of a level set of $\varphi$ is at most $\dim_{\R} N$,
and the real codimension of $N$ is of course $2n-\dim_{\R} N$.
Thus, if we consider the function
\begin{equation}
\psi(z) = \varphi(Az+z_0)+c
\end{equation}
for a unitary $A$ near the identity $I$, small $z_0 \in \C^n$, and small
$c \in \C^{\nu}$, we can ensure that $Z_\psi$ is a complex submanifold
of complex dimension $n-\nu$ through $p$ such that the manifolds $Z_\psi$
and $N$ are transverse as real submanifolds.
We can do this with arbitrarily small $\snorm{A-I}$, $z_0$, and $c$, and hence we can do
so on any relatively compact $W$.
We finish by applying Proposition~\ref{prop:normaltransversality}.
\end{proof}


\section{CR images}

In the following, we consider a holomorphic map $F \colon U \to \C^m$, where
$U \subset \C^n$ is open, and we assume that $N \subset U$ is a real-analytic generic submanifold.
We suppose that $F|_N$ is a diffeomorphism onto its image (as a real-analytic
mapping), and we write $M = F(N)$ for this image.  As noted in the
introduction, any CR image $M$ can be written (locally) in this way.
We suppose that the holomorphic derivative $DF$ is generically of rank $n$.  By this 
we mean that it is of rank $n$ on an open dense subset of $U$, or equivalently (as
$N$ is generic) on a relatively open dense subset of $N$.  As we will see, this requirement ensures that
the CR dimension of $M_{CR}$ is the same as the CR dimension of $N$.

\begin{example}
Let us see what happens if $DF$ is not of rank $n$ generically.  Let $(z,w)$ be coordinates on $\C^2$. Consider
$N = \{z=\bar{w} \}$ and $F(z,w) = (z,0)$.  The CR dimension of $N$ is
$0$, and the CR dimension of $F(N)$ is $1$.
\end{example}

The condition that $DF$
is generically of rank $n$ means that $N$ is the realization of  $M$
with the CR structure extended across the CR singularities.

\begin{prop} \label{prop:locCRsing}
Suppose $N \subset \C^n$ is a connected real-analytic generic submanifold,
$U \subset \C^n$ is
a neighborhood of $N$, $F \colon U \to \C^m$ is
a holomorphic map such that $F|_N$ is a diffeomorphism onto
its image $M=F(N)$, and $DF$ is generically of rank $n$.
Then:
\begin{enumerate}[(i)]
\item
The CR dimension of $M_{CR}$ is the same as the CR dimension of $N$.
\item
Let $p \in N$.  
The submanifold $M$ is CR singular at $F(p)$ if and only if
$DF$ is not of rank $n$ at $p$.
\end{enumerate}
\end{prop}

In particular, note that if $N$ is a hypersurface, then $DF$ is of 
rank $n$ on $N$ if $F|_N$ is a diffeomorphism.  For any $p \in N$, the space $T_p^{0,1}N$ has dimension $n-1$, 
so $T_{F(p)}^{0,1} M$ has dimension at least $n-1$ as well
(e.g., \cite{crext4}*{Proposition 2.6}).
There is only one more real dimension left in $T_{F(p)}M$, so
the dimension of $T_{F(p)}^{0,1} M$ must be $n-1$.  In other words, the
inverse of $F|_N$ on $M$ is a CR diffeomorphism, and $F$ is actually
biholomorphic onto the intrinsic complexification of $M$.
To get interesting CR images, then, one must look at submanifolds of higher codimension.

\begin{proof}
Suppose that $DF$ is of rank $n$ at $p \in N$.  Then locally $F$ takes a
neighborhood $W$ of $p$ in $\C^n$ biholomorphically onto an $n$-dimensional
submanifold of $\C^m$.  Thus $M=F(N)$ near $F(p)$ is a CR submanifold of
the same CR dimension as $N$. Hence, the image $F(W)$ is the intrinsic
complexification of $M$ near $p$.

As $DF$ is generically of rank $n$ in $U$,
it is of rank $n$ on an open dense subset of $N$ because $N$ is generic.
The set $M_{CR}$ is open and
dense in $M$, so the CR dimension of $M_{CR}$ is equal to the CR
dimension of $N$.

Conversely, suppose that $M$ is CR at $F(p)=q$.  The inverse of $F|_N$ near $q$
is a CR diffeomorphism and has a unique extension as a biholomorphism to a
neighborhood of $q$ in the intrinsic complexification of $M$ near $q$.
But the inverse of this biholomorphism must agree with $F$ near $p$,
and hence $F$ near $p$ has rank $n$.
\end{proof}

In the equidimensional case $m=n$, we can relate the CR dimensions
using the transversality condition we introduced.

\begin{prop} \label{prop:locCRsingequitrans}
Suppose $N \subset \C^n$ is a connected real-analytic generic submanifold,
$U \subset \C^n$ is
a neighborhood of $N$, $F \colon U \to \C^n$ is
a holomorphic map such that $F|_N$ is a diffeomorphism onto
its image $M=F(N)$, $\det DF$ vanishes at some $p \in N$, and
\begin{equation}
H_p N \not\subset
C_p Z_{\det DF} .
\end{equation}
Then:
\begin{enumerate}[(i)]
\item
The CR dimension of $M_{CR}$ is the same as the CR dimension of $N$.
\item
Let $p' \in N$.  
The submanifold $M$ is CR singular at $F(p')$ if and only if
$\det DF$ vanishes at $p'$.
\end{enumerate}
\end{prop}

\begin{proof}
As $H_p N \not\subset
C_p Z_{\det DF}$ implies that $\det DF$ does not vanish identically on $N$,
$DF$ must be of rank $n$ generically on $N$.  Therefore,
the conclusions follow from Proposition~\ref{prop:locCRsing}.
\end{proof}


\section{Perturbing away singularities in equidimensional case}
\label{sec:equidim}

The first question we ask is whether we can perturb $F$ to eliminate a CR
singularity.  We first prove that generically we cannot when $m=n$ and when the
CR dimension of $N$ is positive. We do so by finding a geometric obstruction to this
perturbation.

\begin{thm} \label{thm:noperturbmapequi}
Suppose $N \subset \C^n$ is a connected real-analytic generic submanifold,
$U \subset \C^n$ is
a neighborhood of $N$, $F \colon U \to \C^n$ is
a holomorphic map such that $F|_N$ is a diffeomorphism onto
its image $M=F(N)$, and $\det DF$ vanishes at some $p \in N$.
Suppose 
\begin{equation}
H_p N \not\subset
C_p Z_{\det DF} .
\end{equation}

Then for every neighborhood $W$ of $p$ there exist $\epsilon > 0$
and a neighborhood $\widetilde{W} \subset W$ of $p$
such that,
whenever $G \colon W \to \C^n$ is holomorphic
and $\norm{F-G}_{W} < \epsilon$, then
$G|_{N \cap \widetilde{W}}$ is a diffeomorphism onto its image,
and $\det DG$ vanishes somewhere on $N \cap \widetilde{W}$, but not identically on $N \cap \widetilde{W}$.  In particular,
$G(N\cap \widetilde{W})$ is a CR singular submanifold,
and the CR dimensions of $N$ and 
$\left(G(N\cap \widetilde{W})\right)_{CR}$ (and $M_{CR}$) are equal.
\end{thm}

That the CR dimension of $N$ and $M_{CR}$ are equal follows from
Proposition~\ref{prop:locCRsingequitrans}.

\begin{proof}
The proof is simply an application of Lemma~\ref{lemma:main}, using Cauchy's estimates to get a bound for $\norm{\det F-\det G}_{\widetilde{W}}$.  The neighborhood $\widetilde{W}$
may be chosen small enough
so that any such $G|_{N \cap \widetilde{W}}$ is a diffeomorphism.
Note that $H_p N \not\subset
C_p Z_{\det DF}$ implies that $\det DF$ does not vanish identically on $N$,
so neither will $\det DG$ for small enough $\epsilon$
by Cauchy's estimates. Then the conclusion on CR dimensions follows
as in the proof of Proposition~\ref{prop:locCRsingequitrans}.
\end{proof}

\begin{example}
The transversality condition in the theorem is not a necessary condition, only sufficient:
Let $N=\R^2 \subset \C^2$, and let $F(x,y) = (x+iy,x^2+y^2)$.  The image
$F(N)$ is the elliptic Bishop surface $w = \sabs{z}^2$.  It is well known
that an elliptic CR singularity is stable under perturbation, so $F$ may not
be perturbed to eliminate the singularity.  However, in this case $H_0 N =
\{0\}$, so the transversality condition is not satisfied.
\end{example}

\begin{example}
Let us consider an example where the CR singularity can be perturbed away and the transversality condition is not satisfied.
Start again with
$N=\R^2$, and let $F(x,y) = (x+iy,x^2)$.  The image $F(N)$ is $w = (\Re z)^2$,
which is CR singular (a parabolic Bishop surface).
Fix $\epsilon >0$ and let $G(x,y) = (x+iy,x^2+i\epsilon x)$.  Clearly $G$ is arbitrarily
close to $F$, yet the Jacobian determinant of $G$ is $-i(2x+i\epsilon)$, which is
never zero on $N=\R^2$.  Hence, $G(N)$ is a CR submanifold at all points.

By letting $N=\R^2 \times \C$, we create an example of  positive CR dimension.
In this case we let $F(x,y,\xi) = (x+iy,\xi,x^2)$ and (for $\epsilon >0$)
$G(x,y,\xi) = (x+iy,\xi,x^2+i\epsilon x)$.
In this case $H_0 N$ is the $\xi$-direction,
that is, it can be identified with the complex line
$\{ x = 0, y = 0 \}$.
The tangent cone $C_0 Z_{\det DF}$ is given (as a
subset of $\C^3$) by
$\{ -i2x=0 \} = \{ x = 0 \}$, so $H_0 N \subset C_0 Z_{\det DF}$. Again,
$G(N)$ is a CR submanifold at all points.
\end{example}

In both of these examples, the image manifolds themselves cannot be perturbed,
see e.g. Lai~\cite{Lai}; however, the context of the theorem above is somewhat
different and more subtle when the CR dimension is bigger than 0.


\section{Perturbing away singularities in general} \label{section:nonequi}

We start the discussion of the non-equidimensional case with several instructive examples.

\begin{example}
Let $(z,w) \in \C \times \C^2 = \C^3$ be our coordinates.  Let $M$ be given by
\begin{equation}
w_1 = {\abs{z}}^6 , \qquad
w_2 = {\abs{z}}^4 .
\end{equation}
The submanifold $M$ is CR singular, with an isolated CR singularity at the
origin (we will see why in a moment).
The submanifold $M$ at CR points is not generic.  The intrinsic complexification of $M$ is the variety given by $w_1^2 = w_2^3$.  In
particular, the intrinsic complexification is singular, so there is
no biholomorphic change of coordinates on $\C^3$ that realizes $M$
as a subset of $\C^2 \subset \C^3$.

Let $F \colon \C^2 \to \C^3$ be given by
\begin{equation}
F(x,y) = \bigl(x+iy,(x^2+y^2)^3,(x^2+y^2)^2 \bigr) .
\end{equation}
Let $N=\R^2$.  Then $F(N) =  M$, and $F|_{N}$ is a diffeomorphism onto $M$.  
The derivative matrix is
\begin{equation}
DF =
\begin{bmatrix}
1 & i \\
6x(x^2+y^2)^2 & 6y(x^2+y^2)^2 \\
4x(x^2+y^2) & 4y(x^2+y^2)
\end{bmatrix}
\end{equation}
The three subdeterminants are
$6(x^2+y^2)^2(y-ix)$,
$4(x^2+y^2)(y-ix)$, and $0$.  Restricted to $N=\R^2$, these have a common
zero exactly at $0$, so $DF$ (as a holomorphic map on $\C^2$) is of rank $2$ on $N$ except at $0$. By Proposition \ref{prop:locCRsing}, the origin in $\C^3$ is the only CR singular point of $M$. 

For a small $\epsilon > 0$, consider the perturbation of $F$ given by
\begin{equation}
G(x,y) = \bigl(x+iy,(x^2+y^2)^3+\epsilon x,(x^2+y^2)^2 \bigr) .
\end{equation}
Then 
\begin{equation}
DG =
\begin{bmatrix}
1 & i \\
6x(x^2+y^2)^2 + \epsilon & 6y(x^2+y^2)^2 \\
4x(x^2+y^2) & 4y(x^2+y^2)
\end{bmatrix} .
\end{equation}
The three subdeterminants are
$6(x^2+y^2)^2(y-ix)-i\epsilon$,
$4(x^2+y^2)(y-ix)$, and $4y\epsilon(x^2+y^2)$.  Restricted to $N=\R^2$, these have no common
zeros,  so $G(N)$ is a CR submanifold, and it is an arbitrarily
close perturbation of $M$ in our sense.
\end{example}

\begin{example}
We can expand on the previous example by defining $M$ in $\C \times \C^\nu$
by
\begin{equation}
w_1 = \sabs{z}^{p_1} , \quad
\cdots , \quad
w_\nu = \sabs{z}^{p_\nu} ,
\end{equation}
for some numbers $p_1,\ldots,p_\nu$.
The intrinsic complexification of this submanifold is defined by
$w_j^{p_k}=w_k^{p_j}$.  For the proper choice of $p_1,\ldots,p_\nu$,
this subvariety is not contained in any proper nonsingular hypersurface.
This means that we can't biholomorphically ``reduce the dimension'' and
study $M$ in a smaller dimension.

The example also shows that the
biholomorphic equivalence class of $M$ is
not an intrinsic property of $M$ if our information consists of the spaces
$T^{0,1}_p M$.  For this submanifold, $T^{0,1}_p M$ is $\{0\}$ except at the origin,
where its complex dimension is $1$.
This is in fact true for every
CR singular 2-dimensional submanifold with an isolated CR singularity.
\end{example}

\begin{example} \label{example:nonremovablesing}
Define $N \subset \C^3$ by $z_1 = \bar{z}_3$.
This is a generic real-analytic submanifold of codimension 2.  Let $F \colon \C^3 \to \C^4$
be defined by
\begin{equation}
F(z) = (z_1,z_2,z_3^2,z_2z_3) .
\end{equation}
Then $F|_N$ is a diffeomorphism. Also, $DF$ is generically of rank 3, and $F(N)$ has a CR singularity at the origin (and only at the
origin), as we will see shortly.  We claim that, for any neighborhood $U$ of the origin
in $\C^3$,
there exists $\epsilon > 0$ such that if $G \colon U \to \C^4$
is a holomorphic map and $\snorm{F - G}_{U} < \epsilon$, then $DG$ is not of
rank $3$ at some point of $N \cap U$. Thus, the CR singularity cannot be perturbed away.

Compute
\begin{equation}
DF =
\begin{bmatrix}
1 & 0 & 0 \\
0 & 1 & 0 \\
0 & 0 & 2z_3 \\
0 & z_3 & z_2
\end{bmatrix} .
\end{equation}
Let $z=x+iy$.
The tangent space of $N$ at 0 is spanned
by
$\frac{\partial}{\partial x_2}|_0,
\frac{\partial}{\partial y_2}|_0,
\frac{\partial}{\partial x_1}|_0+\frac{\partial}{\partial x_3}|_0,
\frac{\partial}{\partial y_1}|_0-\frac{\partial}{\partial y_3}|_0$.
There are four $3 \times 3$ subdeterminants of $DF$:
$2z_3$, $z_2$, $-2z_3^2$, and $0$.  Thus, the submanifold $S$ where $DF$
is of rank less than 3 is of codimension 2, and $S$ (as a variety) has two of the subdeterminants,
$2z_3$ and $z_2$, as generators at $0$.
We will apply Proposition~\ref{prop:normaltransversality}.
Define $\varphi \colon \C^3 \to \C^2$ by $\varphi(z_1,z_2,z_3)=(2z_3,z_2)$,
that is, using the first two subdeterminants of $DF$.
Then $Z_\varphi=S$, $D\varphi$ is of rank $2$, and the (real) tangent space of $Z_\varphi$ at $0$ is spanned by $\frac{\partial}{\partial x_1}|_0,
\frac{\partial}{\partial y_1}|_0$. Thus,  
 as real submanifolds, $Z_\varphi$ and $N$ are transverse at $0$.
Now let $G$ be a holomorphic map from a neighborhood of $0$ to $\C^4$.
The subvariety of points where $DG$ is of rank less than 3 has codimension at most 2,
as it is the pullback via the components of $DG$ of the
subvariety of $4 \times 3$ matrices not of full rank, and the codimension of that
subvariety is 2.
Let $\psi$ be the map into $\C^2$ given by the first two 
subdeterminants of $DG$ (corresponding to the definition of $\varphi$).
For some ball $U$ centered at the origin, there exists 
$\epsilon > 0$ such that, if $G$ is $\epsilon$-close to $F$,
then via Cauchy's estimates
$Z_\psi$ is still a connected submanifold of a slightly smaller ball
$\widetilde{U}$. But $Z_\psi$ is
generated at $0$ by
perturbations of $2z_3$ and $z_2$, so $Z_\psi$ is a submanifold
of codimension 2.
As the subvariety of points where $DG$ is of rank less than 3 has codimension at most 2,
it must agree with
$Z_\psi$ on $\widetilde{U}$.
We can apply Proposition~\ref{prop:normaltransversality}
to get that $\psi$ has a zero on $N \cap \widetilde{U}$, and therefore
$DG$ is not of rank 3 at some point of $N \cap \widetilde{U}$.
\end{example}

If the codimension $m-n$ of the map is large enough, the CR singularity can
be perturbed away.  In fact, we find the sharp bound on the codimension.

\begin{thm} \label{thm:perturbmap}
Let $N \subset \C^n$ be a connected real-analytic generic  submanifold of
codimension $k$, let
$U$ be a neighborhood of $N$, and let $F \colon U \to \C^m$ be a holomorphic map
such that $F|_N$ is a diffeomorphism and $DF$ is generically of rank $n$.  Suppose that
\begin{equation}
4n-k < 2(m+1) .
\end{equation}
Then for every $p \in N$ there exists a neighborhood $W$ of $p$
such that, for every $\epsilon > 0$, the following holds: There is a holomorphic map $G \colon W \to
\C^m$ such that $\norm{F - G}_{W} < \epsilon$,
$G|_{N\cap W}$ is a diffeomorphism, and
$DG$ is of rank $n$ on $N \cap W$.
In other words, $G(N \cap W)$ is a CR submanifold.
\end{thm}

\begin{proof}
The proof follows the general logic of the proof of Theorem 10.8
in the first edition of Lee~\cite{LeeFirstEd}.
Define $G(z) = F(z)+Az$ for a small $m \times n$ matrix $A$.  Then
$DG(z) = DF(z)+A$.  We want to pick $A$ such that $DG(z)$ is of rank
$n$ for all $z \in N$.
For $j=0,1,\ldots,n-1,$ consider the mapping $\varphi_j \colon N \times M_{j,m\times n}(\C) \to
M_{m\times n}(\C)$ defined by $\varphi_j(z,B) = B-DF(z)$,
where $M_{m \times n}(\C)$ is the complex manifold of $m \times
n$ matrices and $M_{j,m\times n}(\C)$ is the complex submanifold of the matrices of
rank $j$.
If there exists $A$ in the complement of $\varphi_j\bigl(N \times M_{j,m\times n}(\C)\bigr)$
in $M_{m\times n}(\C)$ for all $j=0,1,\ldots,n-1$, then $A$ does the job.

The real dimension of $\varphi_j\bigl(N \times M_{j,m\times n}(\C)\bigr)$
is at most the dimension of $N$ plus the real dimension of 
$M_{j,m\times n}(\C)$ (which is of complex codimension $(n-j)(m-j)$).
That is,
\begin{equation}
\dim_{\R} \varphi_j\bigl(N \times M_{j,m\times n}(\C)\bigr)
\leq (2n-k) + 2\bigl(nm-(n-j)(m-j)\bigr) .
\end{equation}
If that dimension is less than the real dimension of $M_{m\times n}(\C)$,
which is $2nm$, then $A$ as above exists.  That is, $A$ exists
when
for all $j=0,1,\ldots,n-1$,
\begin{equation}
(2n-k) + 2nm-2(n-j)(m-j) < 2nm .
\end{equation}
If the inequality holds for $j=n-1$, then it holds for the
smaller $j$.  So plug $j=n-1$ into the inequality to obtain
$4n-k < 2(m+1)$.
As the images 
$\varphi_j\bigl(N \times M_{j,m\times n}(\C)\bigr)$ are all thin sets in $M_{m\times
n}(\C)$, we can pick $A$ outside of these sets, arbitrarily close to the
zero matrix.
\end{proof}

The inequality in Theorem \ref{thm:perturbmap} is sharp.  That is, if it is not satisfied we can construct a
counterexample.  Example~\ref{example:nonremovablesing} can be generalized to show that, if
$2 \leq k \leq n \leq m$ are integers such that
\begin{equation}
4n-k \geq 2(m+1) ,
\end{equation}
then there exist a connected generic real-analytic submanifold
$N \subset \C^n$ of codimension $k$ with $0 \in N$,
a neighborhood $U$ of $N$, and a holomorphic map $F \colon U \to \C^m$ 
with $F|_N$ a diffeomorphism and  $DF$ generically of rank $n$, having the following property: 
For every neighborhood $W$ of $0$,
there exists  $\epsilon > 0$
such that, for every holomorphic $G \colon W \to \C^m$
with $\norm{F - G}_{W} < \epsilon$,
the derivative
$DG$
is not of rank $n$ somewhere on $N \cap W$.
In other words, $G(N \cap W)$ is not a CR submanifold.

In fact, analogously to Proposition~\ref{prop:genericallyperturbphi}, we can prove
more.  The behavior described above is the generic behavior for $k \geq 2$: Starting
with any $N$ and $F$, if $4n-k \geq 2(m+1)$, we can perturb $F$ slightly to
$\widetilde{F}$ in such a way that $N$ and $\widetilde{F}$ give a 
a CR singular CR image where the CR singularity cannot be perturbed away
by moving $\widetilde{F}$.  First, let us show that such an example exists
to begin with.

\begin{prop} \label{prop:sharpnessconstruction}
Suppose $2 \leq k \leq n \leq m$ are integers such that
\begin{equation}\label{eq:ineqmsmall}
4n-k \geq 2(m+1) .
\end{equation}
Then there exist a connected real-analytic generic submanifold
$N \subset \C^n$ of codimension $k$ with $0 \in N$,
a neighborhood $U$  of $N$, and a holomorphic map $F \colon U \to \C^m$ 
such that $F|_N$ is a diffeomorphism and $DF$ is generically of rank $n$, having the following property:  
For every neighborhood $W$ of $0$,
there exists $\epsilon > 0$
such that, for every holomorphic $G \colon W \to \C^m$
with $\norm{F - G}_{W} < \epsilon$,
the derivative
$DG$
is not of rank $n$ somewhere on $N \cap W$.
In other words, $G(N\cap W)$ is not a CR submanifold.
\end{prop}

\begin{proof}
The idea behind the construction is purely of counting dimensions.
The determinantal variety consisting of all
$m \times n$ matrices not of full rank has codimension $m-n+1$ in the space of
$m \times n$ matrices.  That means if we have only $m-n+1$
complex variables that parametrize $N$, then the set of points of
$N$ where a function $G$ is not of rank $n$ 
should in general be of dimension $2\bigl(n-(m-n+1)\bigr)-k$;
so, as long as this
quantity is nonnegative we should be able to construct an example.

The construction follows the same outline as
Example~\ref{example:nonremovablesing}.
Let us first suppose that $k$ is even.
We also assume without loss of generality that
\begin{equation}
m=2n-\frac{k}{2}-1 ,
\end{equation}
as that is the maximal target dimension allowed by the inequality:
The whole idea is for the matrix $DF$ to be of rank less than $n$ somewhere on
$N$, and by removing rows we cannot possibly make the rank grow.

Let $\ell = n-\frac{k}{2}$, and let $F$ be given by
\begin{equation}
F(z_1,\ldots,z_n)
=
(z_1,\ldots,z_{n-1},
z_{n}^2,
z_{n}z_{n-1},
\ldots ,
z_n z_{n-\ell+1}) .
\end{equation}
The derivative is
\begin{equation}
DF = 
\begin{bmatrix}
1 & 0 & \cdots & 0 & 0 & 0 \\
0 & 1 & \cdots & 0 & 0 & 0 \\
\vdots & \vdots & \ddots & \vdots & \vdots & \vdots\\
0 & 0 & \cdots & 1   & 0   & 0 \\
0 & 0 & \cdots & 0   & 1   & 0 \\
0 & 0 & \cdots & 0   & 0   & 2z_n \\
0 & 0 & \cdots & 0   & z_n & z_{n-1} \\
0 & 0 & \cdots & z_n & 0   & z_{n-2} \\
\vdots & \vdots & \ddots & \vdots & \vdots & \vdots
\end{bmatrix} .
\end{equation}
The set of $m \times n$ matrices of rank less than $n$ is
of codimension $m-n+1=\ell$.
We indeed have $\ell$ subdeterminants, namely 
$2z_n, z_{n-1},\ldots, z_{n-\ell+1}$, that generate at $0$ the submanifold of points 
where the rank of $DF$ is less than $n$.
Let $N$ be the submanifold defined by
\begin{equation} \label{eq:eqforNexample}
z_1=\bar{z}_n, \qquad
z_2=\bar{z}_{n-1}, \qquad
\cdots, \qquad
z_{k/2}=\bar{z}_{n-k/2+1}.
\end{equation}
Now
$\frac{k}{2} = 2n-m-1$ and
$n-\ell+1 = 2n-m$,
and it is again clear that the two submanifolds are transverse at $0$.
The conclusion follows as in Example~\ref{example:nonremovablesing} by appealing to
Proposition~\ref{prop:normaltransversality}.

If $k$ is odd, then $\widetilde{k} = k+1$ satisfies the
inequality \eqref{eq:ineqmsmall}.  We use $\widetilde{k}$ instead of $k$
to construct $F$ and $\widetilde{N}$.  By dropping the
real (or imaginary) part of the last equation in  \eqref{eq:eqforNexample},
we obtain an $N$ of codimension $k$.
The $F$ constructed above is still a diffeomorphism when restricted to $N$.
We conclude as earlier.
\end{proof}

Now let us show that such an example exists near any starting point.

\begin{prop}
Let $N \subset \C^n$ be a connected real-analytic generic  submanifold of
codimension $k \geq 2$,
$U$ a neighborhood of $N$, and $F \colon U \to \C^m$ a holomorphic map
such that $F|_N$ is a diffeomorphism and
$DF$ is generically of rank $n$.
Suppose that
\begin{equation}
4n-k \geq 2(m+1) .
\end{equation}
Then for every $p \in N$ and every $\delta > 0$,
there exist a neighborhood $\widetilde{U}$ of $p$
and a holomorphic map $\widetilde{F} \colon \widetilde{U} \to \C^m$
such that $\bnorm{F-\widetilde{F}} < \delta$, $\widetilde{F}|_N$
is a diffeomorphism,
$D\widetilde{F}$ is generically of rank $n$,
$\widetilde{F}(N \cap \widetilde{U})$ is CR singular
at $\widetilde{F}(p)$,
and $\widetilde{F}$ has the following property:
For every neighborhood $W$ of $p$, with $W \subset \widetilde{U}$,
there exists $\epsilon > 0$
such that, for every $G \colon W \to \C^m$
with $\bnorm{\widetilde{F} - G}_{W} < \epsilon$,
the derivative
$DG$
is not of rank $n$ somewhere on $N \cap W$.
In other words, $G(N \cap W)$ is not a CR submanifold.
\end{prop}

\begin{proof}
The key idea is this: Whether the singularity can be eliminated by perturbing
a map depends only on the quadratic terms of the map.
In fact, we want to
ensure that the $n \times n$ subdeterminants of $D\widetilde{F}$ generate
the ideal of the complex submanifold defined by $\rank D\widetilde{F} < n$
and that this submanifold is of the maximum possible codimension, that is,
the codimension of the set of $m \times n$ matrices of rank less than $n$.
But the $n \times n$ subdeterminants of $D \widetilde{F}$ depend only on terms
of order 2 or less.
With regard to $N$, to make the construction in
Proposition~\ref{prop:sharpnessconstruction} work we only care about transversality
at $p$, and we can make $T_p N$ the same as in the proposition.
Let $X$ be the set of coefficients of order 2 or less (the 2-jet space) for $\widetilde{F}$ such that the
$n \times n$ subdeterminants of $D \widetilde{F}$ do not generate a
submanifold of the same codimension as the space of $m \times n$ matrices,
or generate submanifolds that are not transverse to $N$ at $p$.  This set is
a (real) subvariety of the 2-jet space.  Also, $X$ is not the entire 2-jet space:  In Proposition \ref{prop:sharpnessconstruction} we have demonstrated one point
of the 2-jet space that is not in $X$.
So the complement of $X$ is open and dense, and we are finished.
\end{proof}

For $N \subset \C^n$ of codimension 2,
we can eliminate the singularity in $n+1$ dimensions as long as we start with an equidimensional map and compose it with the natural embedding $\C^n \subset
\C^{n+1}$, that is, if the last component of the resulting map is zero.
In fact, we have the following general statement.

\begin{prop}
Let $N \subset \C^n$ be a connected real-analytic generic  submanifold of
codimension $2\ell$ or $2\ell+1$,
$U$ a neighborhood of $N$, and $F \colon U \to \C^n$ a holomorphic map
such that $F|_N$ is a diffeomorphism
and $DF$ is generically of rank $n$.
Consider $F \oplus 0 \colon U \to \C^n \times \C^\ell$.
Then for every $p \in N$ there exists a neighborhood $W$ of $p$
such that, for every $\epsilon > 0$, there is a holomorphic map $G \colon W \to
\C^n \times \C^\ell$ such that $\snorm{G - F \oplus 0}_{W} < \epsilon$, $G|_{N \cap W}$ is a diffeomorphism,  and
$DG$ is of rank $n$ on $N \cap W$.
In other words, $G(N \cap W)$ is a CR submanifold.
\end{prop}

Note that if $\ell=0$ (e.g., $N$ is a hypersurface),
then $F$ is actually a local biholomorphism already. In the proof we assume that $1 \leq \ell <n$.

\begin{proof}
Suppose that $DF|_p$ is of rank $r$.  After a local holomorphic change of coordinates, write $DF|_p$
as
\begin{equation}
DF|_p =
\begin{bmatrix}
I_{r} & 0 \\
* & 0
\end{bmatrix} .
\end{equation}
As $F|_N$ is a diffeomorphism and $N$ is of real dimension $2n-2\ell$ or
$2n-2\ell-1$, we have that $r \geq n-\ell$.
By the implicit function theorem, there exists a holomorphic change of coordinates
in some small neighborhood $W$ of $p$ such that $p=0$ and
\begin{equation}
F(z) = (z_1,\ldots,z_r,*,\ldots,*) .
\end{equation}
Thus, 
\begin{equation}
DF =
\begin{bmatrix}
I_{n-\ell} & 0 \\
* & *
\end{bmatrix}
\end{equation}
in all of $W$.
For a small $\epsilon \not= 0$, 
the map defined on $W$ in the new coordinates by $G = F \oplus (\epsilon z_{n-\ell+1},\ldots,\epsilon z_n)$
is  close to $F \oplus 0$ in $W$, and $G|_{N \cap W}$ is a diffeomorphism.
Also, one of the $n \times n$
subdeterminants of $DG$ is $\epsilon^\ell$, which is nonzero.
Thus, $DG$ is of rank $n$ on $N \cap W$, and $G(N \cap W)$ is a CR submanifold.
\end{proof}


\section{Quadratic terms for a codimension-2 CR image} \label{section:quad}

In \cite{crext5}*{Corollary 7.2} the authors proved the following.

\begin{thm}\label{thm:codim2CRimagesextn2}
Let $(z,w) \in \C^n \times \C$, $n \geq 2$, be the coordinates and,
near the origin, let $M \subset \C^{n+1}$ be a codimension-2 submanifold given by
\begin{equation}
w = \rho(z,\bar{z}) = 
z^* A z + \overline{z^t B z} + z^t C z +
E(z, \bar{z}),
\end{equation}
where $\rho$ is real analytic,
$A,B,C$ are complex $n \times n$ matrices,  $B$ and $C$ are symmetric,
and $E$ is
$O(\snorm{z}^3)$.
If $M$ is a CR image, then the following hold.
\begin{enumerate}[(a)]
\item
$\rank \begin{bmatrix} A^* \\ B \end{bmatrix} \leq 1$.
\item
   Near the origin, $M$ is biholomorphically equivalent to exactly
   one of the following forms:
   \begin{enumerate}[(1)]
     \item $w = \bar{z}_1 z_2 + \bar{z}_1^2 + O(\snorm{z}^3)$,
     \item $w = \bar{z}_1 z_2 + O(\snorm{z}^3)$,
     \item $w = \sabs{z_1}^2 + a \bar{z}_1^2 + O(\snorm{z}^3)$, $a \geq 0$,
     \item $w = \bar{z}_1^2 + O(\snorm{z}^3)$,
     \item $w = O(\snorm{z}^3)$.
   \end{enumerate}
   Furthermore, examples exist for all five cases.
\end{enumerate}
\pagebreak[2] 
\end{thm}

The purpose of this section is to show that, given any real-analytic CR
structure of codimension-2 type,
we can realize it as a codimension-2 CR image with any one of the five types
of quadratic terms.  Therefore, the extended CR structure of such a CR image
is, as a biholomorphic invariant, independent of the quadratic terms, and vice-versa.

\begin{prop}
Let $(z,w) \in \C^n \times \C$, $n \geq 2$, be the coordinates, and let $Q(z,\bar{z})$ be
one of the five types of quadratic terms from
Theorem~\ref{thm:codim2CRimagesextn2}.  That is, $Q$ is one of the following:
\begin{enumerate}[(1)]
\item $Q(z,\bar{z}) = \bar{z}_1 z_2 + \bar{z}_1^2$,
\item $Q(z,\bar{z}) = \bar{z}_1 z_2$,
\item $Q(z,\bar{z}) = \sabs{z_1}^2 + a \bar{z}_1^2$, $a \geq 0$,
\item $Q(z,\bar{z}) = \bar{z}_1^2$,
\item $Q(z,\bar{z}) = 0$.
\end{enumerate}
Let $N \subset \C^{n+1}$ be a connected real-analytic generic codimension-2 submanifold, and let $p \in N$.

Then there exist a neighborhood $U$ of $p$, a holomorphic mapping $F
\colon U \to \C^{n+1}$, and a real-analytic submanifold $M \subset \C^{n+1}$ defined by
\begin{equation}
w = Q(z,\bar{z}) + O(\snorm{z}^3)
\end{equation}
such that $F(p) = 0$, $F(N \cap U) \subset M$,
$DF$ is of rank $n+1$ on a dense open subset of $N \cap U$,
and $F|_{N \cap U}$ is a
diffeomorphism onto its image.
\end{prop}

\begin{proof}
First we assume that $Q$ is one of the types (1) through (4). Suppose $p=0$,
and let $N$ be defined near $0$ in coordinates
$(\zeta,\omega) \in \C^{n-1} \times \C^2$ by
\begin{equation}
\Im \omega_1 = \rho_1(\zeta,\bar{\zeta},\Re \omega) , \qquad
\Im \omega_2 = \rho_2(\zeta,\bar{\zeta},\Re \omega) ,
\end{equation}
where $\rho_1$ and $\rho_2$ are
$O(\snorm{\zeta}^2+\snorm{\Re \omega}^2)$.
The only barred variable in $Q$ is $\bar{z}_1$, so we write $Q(z,\bar{z}) = Q(z_1,\ldots,z_n,\bar{z}_1)$.
Define $F$ near $0$ by
\begin{align}
z_1 & = \omega_1 + i \omega_2 , \\
z_j & = \zeta_{j-1} , \quad j = 2,\ldots,n , \\
w & = Q(\omega_1 + i \omega_2,\zeta_1,\ldots,\zeta_{n-1},\omega_1 - i \omega_2) .
\end{align}
It is not hard to check that $F|_N$ is locally near 0 a diffeomorphism onto
its image and that $DF$ is of rank $n+1$ on a dense open subset of $N$ near
$0$.  So a CR image $M$ exists.  That $M$ is of the right form follows from
the form of the map.

If $Q$ is of type (5), for the map $F$ defined by the preceding construction
$DF$ does not have the desired rank. We simply modify the construction of
$F$ and replace $Q(z,\bar{z}_1)$ by $\widetilde{Q}(z,\bar{z}_1)=\bar{z}_1^3$
when we define the last component of $F$, so $w=(\omega_1 - i \omega_2)^3$.
\end{proof}

Thus, the type of the quadratic part $Q$ and the extended abstract
CR structure of $M$ are independent biholomorphic invariants.
There are also higher-order invariants.
Let us give an example.

\begin{example}\label{example:HighOrderInvariants}
Consider $M_1$ defined by $w = \bar{z}_1 z_2$ and
$M_2$ defined by $w = \bar{z}_1 z_2 + \bar{z}_1^3$,
as submanifolds of $\C^3$.  Both of these are
CR images of $N = \C \times \R^2$:  The first is the image of
$(\zeta,x,y) \mapsto (x+iy,\zeta,(x-iy)\zeta)$, and the
second is the image of
$(\zeta,x,y) \mapsto (x+iy,\zeta,(x-iy)\zeta+(x-iy)^3)$.
Note that $M_1$ and $M_2$ have the same quadratic
part $\bar{z}_1 z_2$; however, they are not 
equivalent under a biholomorphic change of variables.
This is already
evident by comparing the linear terms of a holomorphic mapping.
Consider near $0$ a holomorphic mapping
\begin{equation}
(a_1 z_1 + a_2 z_2 + a_3 w + \text{h.o.t.}, \quad
b_1 z_1 + b_2 z_2 + b_3 w + \text{h.o.t.}, \quad
c_1 z_1 + c_2 z_2 + c_3 w + \text{h.o.t.})
\end{equation}
that takes $M_2$ into $M_1$, where h.o.t.\ denotes the higher-order terms.  Then
\begin{multline}
c_1 z_1 + c_2 z_2 + c_3 (\bar{z}_1 z_2 + \bar{z}_1^3) + \text{h.o.t.}
\\
=
\overline{\bigl(a_1 z_1 + a_2 z_2 + a_3 (\bar{z}_1 z_2 + \bar{z}_1^3) + \text{h.o.t.}\bigr)}
\bigl(b_1 z_1 + b_2 z_2 + b_3 (\bar{z}_1 z_2 + \bar{z}_1^3) + \text{h.o.t.}\bigr)
\\
=
\bigl(\bar{a}_1 \bar{z}_1 + \bar{a}_2 \bar{z}_2 + \bar{a}_3 z_1 \bar{z}_2 + \bar{a}_3 z_1^3 +
\overline{\text{h.o.t.}}\bigr)
\bigl(b_1 z_1 + b_2 z_2 + b_3 \bar{z}_1 z_2 + b_3 \bar{z}_1^3 + \text{h.o.t.}\bigr).
\end{multline}
Note that the higher-order terms (h.o.t.)\ are all holomorphic (no bars).
We must have $c_1 = c_2 = 0$.  We also must have $c_3 = 0$ as there is no way
we can construct a $\bar{z}_1^3$ on the right side.  So this mapping is
not a biholomorphism.
\end{example}

A natural question therefore is to find all of the higher-order invariants.
However, the higher-order terms of the defining function combine the
extended abstract CR structure (the submanifold $N$) and invariants of the CR singular
embedding.  So part of the question is to disentangle the two.

\begin{question}\label{qn:HighOrderInvariants}
What are the higher-order biholomorphic invariants of codimension-2
CR singular CR images?
\end{question}


\section{Codimension-2 CR images in \texorpdfstring{$\C^3$}{C3}} 
\label{sec:dim3codim1}

In the preceding sections, we studied CR images by writing them as
diffeomorphic images of CR submanifolds under maps holomorphic on the
ambient space. In this section, we return to the intrinsic viewpoint: A CR
image has only removable CR singularities. We study in a specific setting
the extent to which the CR structure near a CR singularity determines
whether the singularity is removable.

Throughout this section, we let $M$ be a connected real-analytic codimension-2
submanifold of $\C^3$, and we assume that the CR dimension of $M_{CR}$ is 1.
After a biholomorphic change of coordinates, $M$ can be written locally near
a CR singularity as 
\begin{equation} \label{eq:formofM}
M:w=Q(z,\bar{z})+E(z,\bar{z}),
\end{equation}
where $0\in M$ is a CR singularity, $Q$ is quadratic, and $E=O(\norm{z}^3)$.
Let $\rho=Q+E$.
Treating $z$ as the coordinates on $M$, we have that $T^{0,1}M_{CR}$ is spanned by
\begin{equation}
    L=\rho_{\bar{z}_2}\frac{\partial}{\partial\bar{z}_1} - \rho_{\bar{z}_1}\frac{\partial}{\partial\bar{z}_2}.
\end{equation}

For $q\in M$, let $\sV_q$ be the closure of $T^{0,1}M_{CR}$ in $\C \otimes TM$
restricted to the fiber at $q$, i.e.,
\begin{equation}
\sV_q = {\overline{T^{0,1}M_{CR}}}^{\C \otimes TM} \vert_{q} .
\end{equation}
Note that if $q$ is a CR point of $M$, then $\sV_q = T^{0,1}_qM_{CR}$.

\begin{prop}
Let $M$ be as above.
If $p$ is a removable CR singularity of $M$,
then $\sV_p$ is a complex line,
meaning a one-dimensional complex vector subspace of $T^{0,1}_p\C^3$.
\end{prop}

\begin{proof}
As $p\in M$ is a removable CR singularity,
$\sV$ is a real-analytic vector bundle near $p$.
The CR dimension of $M_{CR}$ is 1, so $\sV_p$ is a complex line.
\end{proof}

The following example shows that the converse of the above proposition
is false, even if $\sV_q$ is a complex line for all $q \in M$.

\begin{example}
Let $M\subset \C^3$ be given by
\begin{equation}
w=\bar{z}_1z_2 + \bar{z}_2^3.
\end{equation}
The following vector field spans $T^{0,1}M_{CR}$:
\begin{equation}
L=3\bar{z}_2^2\frac{\partial}{\partial\bar{z}_1}-z_2\frac{\partial}{\partial\bar{z}_2}.
\end{equation}
The set of CR singularities of $M$ is $\{z_2=0\}$,
which coincides with the set where $L$ vanishes.

Let $p \in M$ be a CR singularity.
The fiber $\sV_p$ is computed as a set of (topological) limits in $T^{0,1}\C^3$. 
The vector field $L$ divided by $z_2$ extends continuously through
the set of CR singularities,
and 
$\sV_p = \operatorname{span}_{\C} \frac{\partial}{\partial \bar{z}_2}\vert_p$,
a complex line.

If $0\in M$ were a removable CR singularity, then, by \cite{LMSSZ}, there
would exist a real-analytic function on $M$ which is CR on $M_{CR}$ and not
the restriction of a holomorphic function.
However, \cite{crext5}*{Example 3.3} shows that this is not possible.
Hence, $0$ is not a removable CR singularity of $M$.
\end{example}

In Proposition~\ref{prop:notBishopSmallSing}, we give conditions on $M$
under which the CR singularity is not removable.
One of these conditions is the exclusion of $M$ ``corresponding to" a
parabolic Bishop surface.
Informally, this means that the quadratic model of $M$ given by $w=Q$ is not
equivalent to a parabolic Bishop surface (see
Definition~\ref{defn:CorrParBishop}).
The following two examples show that, for submanifolds
corresponding to a parabolic Bishop surface, the removability 
of the CR singularity depends on the error term $E$
and hence on higher-order invariants of $M$.
(In fact, whether $\sV_p$ is a complex line also depends on $E$.)
For a related discussion, see Example \ref{example:HighOrderInvariants} and
Question \ref{qn:HighOrderInvariants}.
\begin{example}
Consider $M\subset\C^3$ defined by
\begin{equation}
w=\rho(z,\bar{z})=z_1\bar{z}_1 + \frac{\bar{z}_1^2}{2} + \frac{i}{2}\big(z_1\bar{z}_1^2+2\bar{z}_1z_2\bar{z}_2\big).
\end{equation}
 This submanifold corresponds to a parabolic Bishop surface, and 
$T^{0,1}M_{CR}$ is spanned by
\begin{equation}
    L
    =i\bar{z}_1z_2\frac{\partial}{\partial\bar{z}_1} - \big(z_1+\bar{z}_1+iz_1\bar{z}_1+iz_2\bar{z}_2\big)\frac{\partial}{\partial\bar{z}_2}.
\end{equation}
The only CR singularity of $M$ is $0$.  Note that $\rho_{\bar{z}_1}$
does not vanish on $M_{CR}$ and $\rho_{\bar{z}_2}/\rho_{\bar{z}_1}$ does not
have a limit as $\C^2\setminus\{0\}\ni z\to 0$.  Hence, $\sV_0$ is not a
complex line, and therefore the CR singularity is not removable.
\end{example}

The next example concerns a submanifold that corresponds to
a parabolic Bishop surface with a removable CR singularity
at $p=0$
(so $\sV_0$ is a complex line).
\begin{example}
Let $M\subset\C^3$ be defined by
$w=\rho(z,\bar{z})=Q(z,\bar{z})+E(z,\bar{z})$, with 
\begin{gather}
    Q = \frac{1}{2}\left(z_1+\bar{z}_1\right)^2 \text{ and }\\
    E = -z_1^2 z_2^2 \bar{z}_1^2 \bar{z}_2^2
    -i z_1^2 z_2^4 \bar{z}_1 \bar{z}_2^4
    + \frac{i}{3} z_1^2 \bar{z}_1^3
    + i z_2^2 \bar{z}_1 \bar{z}_2^2
    + i z_1 z_2^2 \bar{z}_2^2
    + \frac{1}{3} z_1^2 z_2^6 \bar{z}_2^6
    - \frac{1}{2} z_2^4 \bar{z}_2^4.
\end{gather}
Note that
\begin{align}
    \rho_{\bar{z}_1} &= z_1 + \bar{z}_1 -2z_1^2z_2^2\bar{z}_1\bar{z}_2^2
    -iz_1^2z_2^4\bar{z}_2^4
    +i\big(z_1^2\bar{z}_1^2+z_2^2\bar{z}_2^2\big) \text{ and }\\
    \rho_{\bar{z}_2} &= -2z_1^2z_2^2\bar{z}_1^2\bar{z}_2
    -4iz_1^2z_2^4\bar{z}_1\bar{z}_2^3
    +2iz_2^2\bar{z}_1\bar{z}_2
    +2iz_1z_2^2\bar{z}_2
    +2z_1^2z_2^6\bar{z}_2^5
    -2z_2^4\bar{z}_2^3\\
    &= 2iz_2^2\bar{z}_2\rho_{\bar{z}_1}.
\end{align}
Thus,  $\rho_{\bar{z}_1}$ has an isolated zero at $0\in \C^2$, and  ${\rho_{\bar{z}_2}}/{\rho_{\bar{z}_1}}$ extends
to be real analytic near $0$.
We see that $0$ is an isolated CR singularity of $M$, and $T^{0,1}M_{CR}$ is spanned by
\begin{equation}
    L = 2iz_2^2\bar{z}_2\frac{\partial}{\partial\bar{z}_1} - \frac{\partial}{\partial\bar{z}_2}.
\end{equation}
Hence, $0\in M$ is a removable CR singularity.
The extended abstract CR structure of $M$ near $0$ is given by $N\subset
\C^3$ defined by
\begin{equation}
    w=\bar{z}_1 + iz_2^2\bar{z}_2^2.
\end{equation}
\end{example}

We now define the notion of $M$ ``corresponding to'' a parabolic Bishop surface.

\begin{defn}\label{defn:CorrParBishop}
Let $M$ be as above, that is, given in the form \eqref{eq:formofM}.
We say that {\it $Q$ (or $M$) corresponds to a parabolic Bishop surface}
if there exists an invertible complex linear map $T\colon \C^2 \to \C^2$ such that
\begin{equation}
\overline{\partial}Q = T^\ast\big((\zeta_1+\bar{\zeta}_1)\,d\bar{\zeta}_1\big)
\end{equation}
where $(z,\bar{z})$ and $(\zeta,\bar{\zeta})$ are coordinates for
$\operatorname{domain}(T)$ and $\operatorname{range}(T)$ respectively,
and $T^*$ is the pullback.
\end{defn}

Since $M$ is defined by $w=\rho(z,\bar{z}) = Q(z,\bar{z})+E(z,\bar{z})$,
we could absorb or introduce holomorphic functions of $z$ in $\rho$ by
a biholomorphic change of coordinates on $\C^2\times \C$.
In light of this, it is not difficult to see that $Q$ corresponds to
a parabolic Bishop surface if and only if there exist holomorphic
coordinates $(z,w)\in\C^2\times\C$ in which
\begin{equation}
Q(z,\bar{z})=\frac{1}{2}\left(z_1+\bar{z}_1\right)^2.
\end{equation}

The following proposition describes conditions on $M$ under which
a CR singularity is not removable.
One of these conditions is $\overline{\partial}Q\not\equiv 0$.  This means
that $Q$ is not a holomorphic function of $z$ or, equivalently, that there
do not exist holomorphic coordinates on $\C^2\times\C$ in which $Q\equiv 0$.
So, the condition $\overline{\partial}Q\not\equiv 0$ translates to the
quadratic model of $M$, defined by $w=Q$, having a CR singularity at $0$.

\begin{prop}\label{prop:notBishopSmallSing}
Let $M$ be as above, that is, given in the form \eqref{eq:formofM}, with $0$
a CR singularity of $M$.  Suppose that $\overline{\partial}Q\not\equiv 0$,
$Q$ does not correspond to a parabolic Bishop surface, and the real
dimension (as a variety) of the CR singularities, near $0$, is less than
2.  Then $0$ is not a removable CR singularity of $M$.
\end{prop}

\begin{proof}
Suppose $M$ satisfies the hypothesis of this proposition.
We argue by contradiction.
Assume $0$ is a removable CR singularity of $M$.
Then, by a previous result of the authors (see Theorem~\ref{thm:codim2CRimagesextn2}),
$M$ is biholomorphically equivalent, near the origin,
to exactly one of the following forms:
\begin{enumerate}[(1)]
     \item $w = \bar{z}_1 z_2 + \bar{z}_1^2 + O(\snorm{z}^3)$,
     \item $w = \bar{z}_1 z_2 + O(\snorm{z}^3)$,
     \item $w = \sabs{z_1}^2 + a \bar{z}_1^2 + O(\snorm{z}^3)$, $a \geq 0$,
     \item $w = \bar{z}_1^2 + O(\snorm{z}^3)$,
     \item $w = O(\snorm{z}^3)$.
\end{enumerate}

Cases (1) and (2): Suppose $M$ is defined by
\begin{equation}
w= \rho(z,\bar{z})=\bar{z}_1 z_2 + \epsilon\bar{z}_1^2 + E(z,\bar{z})
\end{equation}
where $\epsilon\in\{0,1\}$.
Then $T^{0,1}M_{CR}$ is spanned by
\begin{equation}
    L=\rho_{\bar{z}_2}\frac{\partial}{\partial\bar{z}_1} - \rho_{\bar{z}_1}\frac{\partial}{\partial\bar{z}_2}
    =E_{\bar{z}_2}\frac{\partial}{\partial\bar{z}_1} - \big(z_2+2\epsilon\bar{z}_1+E_{\bar{z}_1}\big)\frac{\partial}{\partial\bar{z}_2}.
\end{equation}
Since $0$ is a removable CR singularity, we must have that
$\rho_{\bar{z}_2} = E_{\bar{z}_2}$ vanishes on
\begin{equation}
    S=\{(z,\rho(z,\bar{z}))\in M\,\colon\, \rho_{\bar{z}_1}=z_2+2\epsilon\bar{z}_1+E_{\bar{z}_1}=0\}.
\end{equation}
So, the set of CR singularities of $M$ is $S$, and $\dim_\R S = 2$ because
the linear parts of $\rho_{\bar{z}_1}$ and $\overline{\rho_{\bar{z}_1}}$ are
linearly independent.

Case (3) is argued similarly when $a\ne 1/2$. If $a=1/2$, then $Q$
corresponds to a parabolic Bishop surface.

Case (4) is handled along the same lines as cases (1) and (2).

Case (5) means $\overline{\partial}Q \equiv 0$.
\end{proof}

It is clear from the discussion so far that, for the CR singularity $0\in
M$, $\sV_0$ is a complex line if and only if
\begin{equation}\label{eqn:division}
\lim\limits_{z\to 0}\, \frac{\rho_{\bar{z}_1}}{\rho_{\bar{z}_2}}
\quad\text{or}\quad
\lim\limits_{z\to 0}\, \frac{\rho_{\bar{z}_2}}{\rho_{\bar{z}_1}}
\quad\text{exists,}
\end{equation}
since $M_{CR}$ has CR dimension 1 and $T^{0,1}M_{CR}$ is spanned by
\begin{equation} \label{eq:theL}
    L=\rho_{\bar{z}_2}\frac{\partial}{\partial\bar{z}_1} - \rho_{\bar{z}_1}\frac{\partial}{\partial\bar{z}_2}.
\end{equation}
The limit considered above is topological.
So, when $0$ is a removable CR singularity of $M$,
the CR structure on $M_{CR}$ extends to an abstract
topological CR structure $\sV$ on $M$.
A necessary and sufficient condition for the singularity to be removable is that
one of the quotients be real analytic.

\begin{prop} \label{prop:quotientanalitic}
Let $M$ be as above, that is,
given in the form \eqref{eq:formofM},
with $0$ a CR singularity of $M$.  Then this CR singularity
is removable if and only if either
\begin{equation}\label{eqn:divisionanalitic}
\frac{\rho_{\bar{z}_1}}{\rho_{\bar{z}_2}}
\quad\text{or}\quad
\frac{\rho_{\bar{z}_2}}{\rho_{\bar{z}_1}}
\end{equation}
is (extends to) a real-analytic function near the origin.
\end{prop}

\begin{proof}
Suppose $0$ is a removable CR singularity.
Then, near 0, there exists a nonvanishing real-analytic vector field
\begin{equation}\label{eq:ltilde}
\widetilde{L} = a \frac{\partial}{\partial\bar{z}_1} + b \frac{\partial}{\partial\bar{z}_2}
\end{equation}
that spans $T^{0,1} M$ at CR points.
Without loss of generality, we may assume that $b \equiv -1$.
At least one of
$\rho_{\bar{z}_1}$ and 
$\rho_{\bar{z}_2}$ is nonzero at CR points, and at those points
\begin{equation}\label{eq:theL2} \rho_{\bar{z}_2} \frac{\partial}{\partial\bar{z}_1} - \rho_{\bar{z}_1}
\frac{\partial}{\partial\bar{z}_2}\end{equation} spans $T^{0,1} M$  
(which is of dimension 1). Since $\widetilde{L}$
also spans $T^{0,1} M$ at CR points,
$\rho_{\bar{z}_1}$ cannot be identically zero.  Thus,
\begin{equation} 
\frac{\rho_{\bar{z}_2}}{\rho_{\bar{z}_1}} \frac{\partial}{\partial\bar{z}_1}
- \frac{\partial}{\partial\bar{z}_2}
\end{equation}
must be equal to $\widetilde{L}$
at all points where $\rho_{\bar{z}_1} \not= 0$.  Then
$\rho_{\bar{z}_2}/\rho_{\bar{z}_1} = a$ at these points, so $\rho_{\bar{z}_2}/\rho_{\bar{z}_1}$ extends to be real analytic near the origin.

For the other direction, recall that the vector field \eqref{eq:theL2}
spans $T^{0,1}M_{CR}$.  If one of the quotients is real analytic, then 
we can divide and get a nonvanishing real-analytic vector field  
giving the extended structure.
\end{proof}

One question is whether we can we ask less of the quotient
than being real analytic.
The following example shows that it is possible for $\sV$
to be a $C^k$ bundle but not a $C^{\ell}$ bundle for any ${\ell}>k$.
In other words, a CR singularity may be removable in a $C^k$ sense but not a
$C^{\ell}$ sense.

\begin{example}
Let $k \geq 0$. Consider $M$ given by
\begin{equation}
w=\bar{z}_1 z_2 + \bar{z}_2^{k+3}.
\end{equation}
Then $\rho_{\bar{z}_2}/\rho_{\bar{z}_1}$ is a $C^k$ function near
$z=0$ that vanishes at $z=0$, but it is not $C^{\ell}$ for ${\ell}>k$.
\end{example}

In contrast, if $\sV$ is a $C^\infty$ bundle on $M$,
then it is, in fact, a $C^\omega$ bundle.
First we need a small lemma.

\begin{lemma} \label{lemma:smoothdivision}
Suppose that $\alpha$ and $\beta$ are two real-analytic functions defined on
some connected open set,
$\beta$ is not identically zero,
and $\alpha/\beta$ is smooth, in the sense that there exists a $C^\infty$
function $\gamma$ such that $\alpha = \gamma \beta$.  Then $\gamma$ is real analytic.
\end{lemma}

\begin{proof}
Fix a point $p$.
After a rotation we apply the
Weierstrass preparation theorem to $\beta$; 
after dividing by a unit we may,
without loss of generality, assume that  
$\beta$ is a Weierstrass polynomial.
Let
$\widetilde{\alpha}$,
$\widetilde{\beta}$,
$\widetilde{\gamma}$ denote the formal power series at $p$.
As formal power series
we have
\begin{equation}
\widetilde{\alpha}
=
\widetilde{\gamma}
\widetilde{\beta} .
\end{equation}
We can also apply the real-analytic Weierstrass division theorem to
find a convergent power series $q$ and a remainder $r$ such that
$\alpha = q \beta + r$.  Since the
series $q$ and $r$ are unique and are the same as if we had done this division
formally, we find that $r=0$ and $q=\widetilde{\gamma}$.  In particular, the series for 
$\widetilde{\gamma}$ converges.  Thus, from the real-analytic division
theorem, we find that $\alpha = q \beta$ as functions.  Hence, $q=\gamma$ on
the set where $\beta \not= 0$, and as that is a thin set we find that
$q=\gamma$ on a neighborhood of $p$.
\end{proof}

\begin{prop}
Let $M$ be as above, that is, given in the form \eqref{eq:formofM},
with $0$ a CR singularity of $M$.
Suppose there exists a $C^\infty$ bundle $\sV$ near $0$
that agrees with the CR structure of $M_{CR}$. Then $\sV$
is a $C^{\omega}$ bundle near $0$, that is, $0$ is a removable CR singularity.
\end{prop}

\begin{proof}
Let us work locally near $0$.
Let $\widetilde{L}$ be the vector field \eqref{eq:ltilde},
but now with $a$ and $b$ being $C^\infty$.  Again assume $b\equiv -1$.
As in the proof of 
Proposition~\ref{prop:quotientanalitic},
we find that  $\rho_{\bar{z}_1}$ cannot be identically zero and
$\rho_{\bar{z}_2}  = a \rho_{\bar{z}_1}$.
We apply Lemma~\ref{lemma:smoothdivision} to find that $a$ is real analytic,
and therefore $0$ is a removable CR singularity.
\end{proof}


\def\MR#1{\relax\ifhmode\unskip\spacefactor3000 \space\fi%
  \href{http://mathscinet.ams.org/mathscinet-getitem?mr=#1}{MR#1}}


\end{document}